\documentclass[11pt,onecolumn]{IEEEtran}

\usepackage{setspace}

\usepackage{multicol}
\usepackage{psfig}
\usepackage{subfigure}
\usepackage{amsmath}
\usepackage{amssymb}
\usepackage{amsfonts}
\usepackage{graphicx}
\usepackage{url}
\usepackage{ccaption}
\usepackage{booktabs} 
\usepackage{color}
\usepackage{comment}
\usepackage{subfig}
\usepackage{epstopdf}
\usepackage{cite}

\ifCLASSINFOpdf
\else
\fi

\newcommand{\diag}{{\rm diag}}

\newcommand{\tr}{{\rm tr}}

\newtheorem{Def}{Definition}
\newtheorem{Thm}{Theorem}
\newtheorem{Asu}{Assumption}

\newtheorem{Lem}{Lemma}
\newtheorem{Rem}{Remark}
\newtheorem{Step}{Step}

\newtheorem{Prm}{Problem}
\begin{document}
%
\title{Leader-based Optimal Coordination Control for the Consensus Problem of Multi-agent Differential Games via Fuzzy Adaptive Dynamic
Programming}

%
%
%

\author{Huaguang~Zhang,~
        Jilie~Zhang, Guang-Hong~Yang,~
        and~Yanhong ~Luo
\thanks{Huaguang~Zhang, Jilie Zhang, Guang-Hong Yang and Yanhong Luo are with the College of Information Science and Engineering, Northeastern University, Shenyang, Liaoning, 110819, P.\,R. China (hgzhang@ieee.org; jilie0226@163.com; yangguanghong@ise.neu.edu.cn; neuluo@163.com). Huaguang Zhang and Guang-Hong~Yang are also with State Key Laboratory of Synthetical Automation for Process Industries (Northeastern University), Shenyang, Liaoning, 110004, China.}

}

%
%

\markboth{IEEE TRANSACTIONS ON FUZZY SYSTEMS, VOL. XX, NO. X, FEBRUARY 20XX}%
{ZHANG \MakeLowercase{\textit{et al.}}: LEADER-BASED OPTIMAL COORDINATION CONTROL FOR THE CONSENSUS PROBLEM}
%

\maketitle

\begin{abstract}
In this paper, a new online scheme is presented to design the optimal coordination control for the consensus problem of multi-agent differential games by fuzzy adaptive dynamic programming (FADP), which brings together game theory, generalized fuzzy hyperbolic model (GFHM) and adaptive dynamic programming (ADP). In general, the optimal coordination control for multi-agent differential games is the solution of the coupled Hamilton-Jacobi (HJ) equations. Here, for the first time, GFHMs are used to approximate the solutions (value functions) of the coupled HJ equations, based on policy iteration (PI) algorithm. Namely, for each agent, GFHM is used to capture the mapping between the local consensus error and local value function. Since our scheme uses the single-network architecture for each agent (which eliminates the action network model compared with dual-network architecture), it is a more reasonable architecture for multi-agent systems. Furthermore, the approximation solution is utilized to obtain the optimal coordination control. Finally, we give the stability analysis for our scheme, and prove the weight estimation error and the local consensus error are uniformly ultimately bounded (UUB). Further, the control node trajectory is proven to be cooperative uniformly ultimately bounded (CUUB).
\end{abstract}


\begin{IEEEkeywords}
Optimal coordination control, Consensus problem, Multi-agent
differential game, Fuzzy adaptive dynamic programming, Generalized
fuzzy hyperbolic model.
\end{IEEEkeywords}

%
\IEEEpeerreviewmaketitle



\section{Introduction}
In recent decades, the consensus problems of multi-agent systems (for instance, formation control \cite{Lin2005}, flocking \cite{Lee2006,Zhu2013}, rendezvous \cite{Lin2004} and sensor networks \cite{Xiao2005,Cao2013} and so on) have received
considerable attention, such as \cite{Meng2013a,Trecate2006} and \cite{Meng2013}. In the early days, consensus problems originated
from computer science and formed the foundation of the field of distributed computing \cite{Lynch1996}. Subsequently, these problems
were developed to management science and statistics \cite{DeGroot1974}. Now, references \cite{Borkar1982} and \cite{Tsitsiklis1984} in 1980s are referred to as the pioneering work on consensus problems for control theory. In \cite{Saber2004}, Olfati-Saber and Murray presented the fundamental framework for solving consensus problems for multi-agent systems. And overviews \cite{Ren2005} and \cite{Saber2007} have summarized the recent achievements of coordination control for consensus problems of multi-agent systems.

In \cite{Ren2005}, for consensus problems, Ren et al. proposed an open research problem, that is, how to design the optimal coordination control, which not only makes multi-agent systems stable, but also minimizes their performance indexes. In a physical sense, the optimal coordination control makes every agent use up the least amount of energy, and makes them reach a consensus. In fact, every agent depends on the actions of itself and all its neighborhood agents. Therefore, every agent requires to choose a control to minimize its own performance index by acting on itself, according to the outcomes of its neighborhood agents. It is similar to the multi-player cooperative game.

Game theory \cite{Tijs2003} studies strategic decision making problems. More formally, it is ``the study of mathematical models of conflict and cooperation between intelligent rational decision-makers." In general, if it is cooperative games, the communication among players is allowed. The decision for each player depends on the actions of himself and all the other players. In the early days, game theory was used widely for solving the problem of multi-player games, such as, \cite{Wei2011} and \cite{Vamvoudakis2011}. Recently, game theory has also become the theoretical basis in the field of multi-agent games in \cite{Vamvoudakisa2012,Kazerooni2009,Parsons2002}. The evolution of the agents' state variables is governed by differential equations. The problem of finding an optimal strategy in a differential game is closely related to the optimal control theory. In particular, the closed-loop strategies can be found by Bellman's dynamic programming method, such as \cite{Wei2011,Vamvoudakis2011,Vamvoudakisa2012}. For multi-agent systems, since every agent's action depends on the outcomes of itself and all the neighborhood agents, the coupled Hamilton-Jacobi (HJ) equations are set up. Therefore, for
multi-agent differential games, the optimal coordination control relies on solving the coupled HJ equations. However, in general, it
is very difficult.

Therefore, in this paper, ADP algorithm (\cite{Zhang2011} and \cite{Wei2012}), which combines adaptive control and reinforcement learning, is introduced to learn the solution of HJ equations online for multi-agent systems. The excellent overview of the state-of-the-art developments of ADP algorithm has been presented in \cite{Kirk2004,Si2004,Wang2004,Lewisi2009}. How to approximate the value function is a key problem in the ADP algorithm. Based on Weierstrass higher-order approximation theorem \cite{Rudin1976}, we know that $N$ complete basis can be used to approximate the solution of the Hamilton-Jacobi-Bellman (HJB) equation by linear expression, as $N \rightarrow \infty$. For finite $N$, however, the approximation theorem will be sensitive to the chosen basis. If a smooth function can not be spanned by finite $N$ independent basis sets, then the group of basis sets will not be able to strictly approximate the function. Therefore, we want to choose a group of independent basis as better as possible to capture the significant features of the value function. Traditionally, neural networks are used as the approximator for it. However, neural networks do not have the clear physical significance, and activation functions (basis functions) are manually chosen. So we do not know whether the selected activation function is appropriate. It motives us to circumvent the disadvantage by using fuzzy approximation technology (fuzzy approximator). The fuzzy approximation technology can characterize the value function more reasonably by the knowledge from human experts and experiments. The generalized fuzzy hyperbolic model (GFHM) is a better selection as a function approximator \cite{Zhang2001a,Zhang2006,Zhang2001} which has clear physical significance \cite{Margaliot2003} (It is easy to construct an GFHM if we know some linguistic information about the relationship between function output and the input variables), and the model weights can be optimized by adaptive learning. Specially, GFHM transforms the problem (that is, how to choose basis functions in neural network model) into how to translate the input variables. In this way, the entire input space can be covered as much as possible by choosing sufficient and proper generalized input variables. So, GFHM is a better approximator for estimating the value function, such as \cite{Zhang2012a} and \cite{Zhang2013}.

In recent years, some optimal control methods have been proposed for the multi-agent consensus problem, such as the linear quadratic regulator (LQR) technology \cite{Zhang2012} and the model predictive control (MPC) technology \cite{Trecate2009}. However, the method in\cite{Zhang2012} is only limited to the linear systems and is an off-line design procedure. Though the method in \cite{Trecate2009} has obtained a good on-line controller for single- and double-integrator multi-agent systems (specially, the time-varying communication network), the continuous sampling and real-time predictive processes are required, and the method gets a control sequence for the finite horizon. By the way, \cite{Trecate2009} addresses the case that agents are discrete-time systems with leaderless. Here, we deal with the
continuous nonlinear consensus problem with the leader online through using the ADP algorithm. The algorithm can solve the coupled
HJ equations directly by the policy iteration and adaptive control methods, and simultaneously avoiding the sampling and repeated predictive processes in \cite{Trecate2009}. In addition, we get an optimal function relationship of control for the infinite horizon,
when the ADP algorithm does not change the adjustable weight of control.

In this paper, our major idea is to utilize game theory to solve the optimal coordination control problem for multi-agent systems based
on adaptive dynamic programming. By Bellman's dynamic programming method, we construct the coupled HJ equations for multi-agent
differential games. To obtain the solution of the coupled HJ equations, GFHMs are used to approximate the value functions (solution) under the framework of PI algorithm \cite{Bertsekas1996}. It results in the errors of the coupled HJ equations. To minimize the errors resulting from GFHM approximators, the gradient descent is used to update weights of these GFHM approximators. The update of weights is implemented continuously until they do not change. We call it fuzzy adaptive dynamic programming (FADP). Finally, we analyze the stability conditions and prove the weight error and the local consensus error are uniformly ultimately bounded (UUB).

The contributions of the paper include:
\begin{enumerate}
   \item[1.] The cooperative problem of multi-player games is developed to the coordination consensus control problem of nonlinear multi-agent systems. The paper builds a relationship between the optimal consensus problem for multi-agent systems and Nash equilibrium of cooperative game theory.
   \item[2.] The coupled Hamilton-Jacobi equations for multi-agent systems are established by Bellman's dynamic programming, and then the stability analysis is developed for our scheme.
   \item[3.] The open problem, i.e., the optimal consensus problem for multi-agent systems presented in \cite{Ren2005}, is solved by fuzzy adaptive dynamic programming with single-network architecture for the first time. Namely, only one GFHM is used to approximate the local value function for each agent.
   \item[4.] The proposed single-network architecture eliminates the action network model and reduces the number of updated weights, compared with the dual-network architecture (that in \cite{Zhang2011} and \cite{KyriakosG2010}).
\end{enumerate}

The rest of this paper is organized as follows. In section \ref{sec2}, some definitions and notions are given. The local consensus dynamic error system is established in section \ref{sec3}. The coupled Hamilton-Jacobi equations for multi-agent systems are deduced, the stability of Nash equilibrium is proven and the coupled HJ equations are solved by PI algorithm in section \ref{sec4}. Section \ref{sec6} derives the approximation coupled HJ equations by using GFHMs. Section\ref{sec7} gives stability analysis for our scheme and proves the weight estimation error and the local consensus error are UUB, and the control node trajectory is CUUB. Finally, a numerical example is given to illustrate the effectiveness of our scheme.

\section{Preliminaries}\label{sec2}
The purpose of this section is to provide the foundations of graph theory, information consensus and generalized fuzzy hyperbolic model.
\subsection{Graph Theory}
In this paper, graph theory is used to analyse the multi-agent systems as a very
helpful mathematical tool. Regardless of the unidirectional
information flow or bidirectional one, the topology of a
communication network can be expressed by a weighted graph.

Let $\mathcal{G} = (\mathcal{V}, \mathcal{E}, \mathcal{A})$ be a
weighted graph of $N$ nodes with the nonempty finite set of nodes
$\mathcal{V}=\{{v}_1,\ldots,{v}_N\}$, where set of edges
$\mathcal{E}$ belongs to the product space of $\mathcal{V}$  (i.e.
$\mathcal{E} \subseteq \mathcal{V} \times \mathcal{V}$), an edge of
$\mathcal{G}$ is denoted by $e_{ij}=(v_j,v_i)$, which is a direct
path from node $j$ to node $i$, and $\mathcal{A}=[a_{ij}]$ is a
weighted adjacency matrix with nonnegative adjacency elements, i.e.,
$a_{ij} \geq 0$, $e_{ij} \in \mathcal{E} \Leftrightarrow a_{ij} >
0$, otherwise $a_{ij} = 0$. The node index $i$ belongs to a finite
index set $\mathcal{I}=\{1,2,\ldots,N\}$.

\begin{Def}[Laplacian Matrix]\label{D1} The graph Laplacian matrix $L=[l_{ij}]$ is defined
as $\mathcal{D}-\mathcal{A}$, with $\mathcal{D} = \diag \{d_i\} \in
R^{N \times N}$ being the in-degree matrix of graph, where $d_i =
\sum_{j=1}^N a_{ij}$ is in-degree of node $v_i$ in graph.
\end{Def}
\begin{Rem}\label{R1}
Laplacian matrix has all row sums equal to zero.
\end{Rem}

In this paper, we assume the graph is simple, e.g. no repeated edges
and no self loops. The set of neighbors of node $v_i$ is denoted by
$N_i=\{v_j\in\mathcal{V}:(v_j,v_i)\in\mathcal{E}\}$. A graph is
referred to as a spanning tree, if there is a node $v_i$ (called the
root), such that there is a directed path from the root to any other
nodes in the graph. A digraph is said to be strongly connected, if
there is a directed path from node $i$ to node $j$, for all distinct
nodes $v_i, v_j \in \mathcal{V}$. A digraph has a
spanning tree if it is strongly connected, but not vice versa.

Here, we focus on the strongly connected communication digraph with fixed topology.

\subsection{Consensus for Networks of Agents}
A multi-agent system is a network which consists of a group of
agents. Every agent is called as a node in network. Let $x_i \in
{R}^n$  denote the state of node $v_i$. We call $\mathcal{G}_x =
(\mathcal{G},x)$ (with the state $x \in {R}^{N n}$) a network (or
algebraic graph), where $x=[x_1^T,\ldots,x_N^T]^T$. The state of a
node might represent the physical quantity of the agent, such as
altitude, velocity, angle, voltage and so on. We say nodes of a
network have reached a consensus if and only if $x_i=x_j$ for all
$i,j \in \mathcal{I}, i\neq j$. For the consensus problem with
leader, every node requires $x_i(t) \rightarrow x_0(t)$, $\forall i
\in \mathcal{I}$, where $x_0(t)$ is state trajectory of the leader.

\subsection{Generalized Fuzzy Hyperbolic Model}
\begin{Def}\label{D2}
Given a plant with $n$ input variables $x=[x_1(t),\ldots, x_n(t)]^T$
and an output variable $y$. We call the fuzzy rule base the
generalized fuzzy hyperbolic rule base if it satisfies the following
conditions:
\begin{enumerate}
\item The $l_{th}$ fuzzy rule takes the following form $(l= 1 ,\ldots, 2^m)$:
\begin{eqnarray*}
R^l: &&\mbox{IF}\  (x_1-d_{11})\  \mbox{is}\  F_{x_{11}},  \ldots,(x_1-d_{1w_1})\ \mbox{is}\ F_{x_{1w_1}} \mbox{,} (x_2-d_{21})\mbox{is}\  F_{x_{21}}\mbox{,} \ldots \mbox{,}\nonumber \\ 
&&(x_2-d_{2w_2})\ \mbox{is}\ F_{x_{2w_2}} \mbox{,}\ldots \mbox{,} (x_n-d_{n1})\  \mbox{is}\  F_{x_{n1}} \mbox{,} \ldots \mbox{, and} (x_n-d_{nw_n})\ \mbox{is}\ F_{x_{nw_n}}. \nonumber \\
&&\mbox{THEN}\ \ y^l = c_{F_{{11}}} + \ldots + c_{F_{{1w_1}}}+\ldots+ c_{F_{{n1}}} + \ldots + c_{F_{{nw_n}}},  \nonumber
\end{eqnarray*}
where, $w_z (z=1, \ldots, n)$ represents the number of
transformations associated with each $x_z$, and $d_{zj}(z =
1,\ldots, n, j=1,\ldots ,w_z)$ are constants that define the
transformations, $F_{x_{zj}}$ are fuzzy sets of $x_z-d_{zj}$ which
include subsets $P_z$ (positive) and $N_z$ (negative), and
$c_{F_{zj}}$ are constants
corresponding to $F_{x_{zj}}$. \\
\item The constants $c_{F_{zj}} (z = 1, \ldots, n, j = 1, \ldots,w_z)$ in the THEN-part correspond
to $F_{x_{zj}}$ in the IF-part. That is, if there is $F_{x_{zj}}$ in
the IF-part, $c_{F_{zj}}$ must appear in the THEN-part; otherwise,
$c_{F_{zj}}$ does not appear in the THEN-part.\\
\item There are $s = 2^m$ fuzzy rules in the rule base, where $m = \sum ^n_{i=1}w_i$ that is,
all the possible $P_z$ and $N_z$ combinations of input variables in
the IF-part and all the linear combinations of constants in the
THEN-part.
\end{enumerate}
\end{Def}
\begin{Lem}\cite{Zhang2006,Zhang2001a}\label{L1}
For a multiple input single output system, $y = f(x_1, x_2,\ldots,
x_n)$, define the generalized input variables as
$$\bar{x}_i=x_z-d_{zj},\quad i=1,\ldots,m $$
and the generalized fuzzy hyperbolic rule base as in Definition \ref{D2}, respectively, where the membership functions of the
generalized input variables {$P_z$} and {$N_z$} are defined as
\begin{align*}
\mu_{P_{z}}({x}_z)=e^{-\frac{1}{2}({x}_z-\phi_z)^2},\quad \mu_{N_{z}}({x}_z)=e^{-\frac{1}{2}({x}_z+\phi_z)^2},
\end{align*}
where $\phi_z >0$.

We can then derive the following model:
\begin{eqnarray*}
y=\theta^T \tanh(\Phi \bar{x})+\zeta,
\end{eqnarray*}
where $\theta=[\theta_1,\ldots,\theta_m]^T$ is an ideal vector; $\tanh(\Phi \bar{x})=[\tanh(\phi_1 \bar{x}_1), \ldots , \tanh(\phi_m
\bar{x}_m)]^T$ with $\Phi = \diag\{\phi_i\}$ ($i=1,\ldots,m$) and $\bar{x}=[\bar{x}_1,\ldots,\bar{x}_m]^T$; and $\zeta$ is a constant scalar. We call it as generalized fuzzy hyperbolic model (GFHM).
\end{Lem}

\begin{Lem}\cite{Zhang2006}\label{L2}
Let $F$ be the set of all generalized fuzzy hyperbolic model given
by Lemma \ref{L1}. For any given real continuous function $f(x)$ on
the compact set ${U}\subset{R}^n$ and any arbitrary $\delta > 0$,
there exists a $h(x)\in{F}$ such that $$\sup\limits _{x \in
{U}}|f(x)-h(x)|<\delta.$$
\end{Lem}

\begin{Rem}
Lemma \ref{L2} shows that
GFHM can uniformly approximate any nonlinear function over $U$ to
any degree of accuracy if $U$ is compact, that is, the GFHM is a
universal approximator (see \cite{Zhang2006} for details).
Therefore, GFHM can approximate the function with error bound, by
sufficient and proper generalized input variables which cover the
entire space as much as possible. Here, the sufficient and
proper translational quantity of input variables requires to be chosen by
expertise or manual selection.
\end{Rem}

\section{Consensus error dynamic system}\label{sec3}
Consider multi-agent systems with $N$ agents in the form of
communication network $\mathcal{G}_x$. Their node dynamics are
\begin{eqnarray}\label{E1}
\dot{x}_i=f(x_i)+g_i(x_i)u_i,
\end{eqnarray}
where $x_i(t) \in R^n$ is the state of node $v_i$, $u_i(t)\in
R^{m_i}$ is the input coordination control. $f(x_i) \in R^n$ and
$g_i(x_i)\in R^{n\times m_i}$, such that $f(0)=0$ and $f+g_iu_i$ contains the origin ($\|g_i(x_i)\|<\beta_i$, $\|\cdot\|$
is the Euclidean norm).

The global network dynamics is
\begin{eqnarray}\label{E2}
\dot{x}=f(x)+g(x)u,
\end{eqnarray}
where the global state vector of the multi-agent system (\ref{E2})
is $x = [x_1^T, x_2^T, \cdots, x_N^T]^T \in R^{Nn}$, the global
nodedynamics vector is $f(x) = [f^T (x_1), f^T (x_2), \cdots, f^T
(x_N)]^T \in R^{Nn}$, $g(x) = \diag \{g_i (x_i)\} \in R^{Nn\times
M}$ with $i \in \mathcal{I}$ and the global control input $u =
[u_1^T, u_2^T,\cdots, u_N^T ]^T \in R^M$($M=m_1+\cdots+m_N$). $N$ is
the number of the nodes.

The state of the control node (or leader) is $x_0(t)$ which
satisfies the dynamics
\begin{eqnarray}\label{E3}
\dot{x}_0=f(x_0),
\end{eqnarray}
where $x_0(t) \in R^n$, $f(x_0)$ is the differentiable function.

The local neighborhood consensus error $e_i$ for node $v_i$ is
defined as
\begin{eqnarray}\label{E4}
e_i=\sum \limits_{j\in N_i}a_{ij}(x_i-x_j)+b_i(x_i-x_0),
\end{eqnarray}
where $e_i=[e_{i1},e_{i2},\ldots,e_{in}]^T$ ($e_i\in R^n$). $b_i$ is
the pinning gain ($b_i \geq 0$). Note that $b_i > 0$ for at least
one $i$. Then $b_i = 0$ if and only if there is not a direct path
from the control node to the $i_{th}$ node in $\mathcal{G}_x$;
otherwise $b_i > 0$. The nodes $v_i$ ($b_i \neq 0$) are referred to
as the pinned or controlled nodes.

\begin{Rem}\label{R3} The local neighborhood consensus error $e_i$ represents the
information whether node $v_i$ agrees on the leader and its
neighbors, that is, whether the multi-agent system reach a
consensus, $e_i \rightarrow 0$ as $t \rightarrow \infty$.
\end{Rem}

The global error vector for the network $\mathcal{G}_x$ is
\begin{align}\label{E5}
e=&(L \otimes I_n)x + (B \otimes I_n)(x-\underline{x}_0)\nonumber\\
=&(L \otimes I_n)x-(L \otimes I_n)\underline{x}_0+(B \otimes
I_n)(x-\underline{x}_0)\nonumber\\
=&((L+B) \otimes I_n)(x-\underline{x}_0)\nonumber\\
=&\mathcal{L}(x-\underline{x}_0),
\end{align}
with $\mathcal{L}=(L+B) \otimes I_n$ ($I_n$ is an identity matrix
with $n$ dimensions), where $L$ is the Laplacian matrix for the
network $\mathcal{G}_x$; $e = [e_1^T, e_2^T, \cdots, e_N^T]^T$ $\in
R^{Nn}$ and $\underline{x}_0 = \underline{\emph{I}}x_0 \in R^{Nn}$,
with $\underline{\emph{I}}=\underline{1} \otimes \emph{I}_n \in
R^{Nn \times n}$ and $\underline{1}$ is the N-vector of ones;
$B=[b_{ij}] \in R^{N \times N}$ is a diagonal matrix with diagonal
entries $b_i$ (i.e. $b_{ii}=b_i$). $\otimes$ is the Kronecker
product operator. Differentiating (\ref{E4}) or (\ref{E5}), the
dynamics of local neighborhood consensus error for network
$\mathcal{G}_x$ are given by
\begin{align}\label{E6}
\dot{e}_i=&((L_i+B_i) \otimes
I_n)(\dot{x}-\dot{\underline{x}}_0)\nonumber\\
=&((L_i+B_i) \otimes
I_n)(f(x)+g(x)u-\underline{f}(x_0))\nonumber\\
=&((L_i+B_i) \otimes I_n)(f_e(t)+g(x)u)\nonumber\\
=& \sum_{j\in \mathcal{I}}((l_{ij}+b_{ij})\otimes
I_n)(f_{ej}(t)+g_j(x_j)u_j)\nonumber\\
 =&((l_{ii}+b_{ii})\otimes I_n)(f_{ei}(t)+g_i(x_i)u_i) + \sum_{j\in N_i}((l_{ij}+b_{ij})\otimes I_n)(f_{ej}(t)+g_j(x_j)u_j),
\end{align}
where $f_e(t)=f(x)-\underline{f}(x_0)$ with
$\underline{f}(x_0)=\underline{I}f(x_0)$,
$f_{ei}(t)=f(x_i)-{f}(x_0)$ and $f_{ej}(t)=f(x_j)-{f}(x_0), (j \in
N_i)$. $L_i$ is denoted as a row vector which is the $i_{th}$ row
vector of the Laplacian matrix $L$, that is, $L_i=[l_{i1}, \ldots,
l_{ii}, \ldots, l_{iN}]$. Similarly, $B_i=[b_{i1}, \ldots,
b_{ii},\ldots b_{iN}]$.
\begin{Rem}\label{R4}
Since $a_{ij}$ is zero when the node $v_j$ is not the neighbor of
node $v_i$, the expressions (\ref{E6}) only contain control inputs
of all the neighbors of node $v_i$ and itself in network
$\mathcal{G}_x$. In fact, it denotes that the local neighborhood
consensus error depends on the states and the control inputs from
node $v_i$ and all of its neighbors.
\end{Rem}
\begin{Def}(Uniformly Ultimately Bounded (UUB))
The local neighborhood consensus error $e_i(t)\in R^n$ is uniformly
ultimately bounded (UUB) if there exists a compact set $\Omega_i \in
R^n$ so that $\forall e_i(t_0) \in \Omega_i$  there exists a bound
$\mathcal{B}_i$ and a time $t_{fi} (\mathcal{B}_i, e_i(t_0))$, both
independent of $t_0\geq 0$, such that $\|e_i(t)\|\leq \mathcal{B}_i
\  \forall t\geq t_0 + t_{fi}$.
\end{Def}
\begin{Def}(Cooperative Uniformly Ultimately Bounded (CUUB))\cite{Das2010}\label{D6}
The control node trajectory $x_0(t)$ given by (\ref{E3}) is
cooperative uniformly ultimately bounded (CUUB) with respect to
solutions of node dynamics (\ref{E1}) if there exists a compact set
$\Omega \subset R^n$ so that $\forall (x_i(t_0)-x_0(t_0)) \in
\Omega$, there exist a bound $\mathcal{C}$ and a time $t_f
(\mathcal{C}, (x_i(t_0)- x_0(t_0)))$, both independent of $t_0\geq
0$, such that $ \|x_i(t)- x_0(t)\| \leq \mathcal{C}$ $\forall i,
\forall t \geq t_0 + t_f $.
\end{Def}

\section{Optimal Coordination Control}\label{sec4}
To reach a consensus while simultaneously minimizing the local
performance index of every agent, we use the machinery of $N$-person
cooperative games (\cite{Vamvoudakis2011,Vamvoudakisa2012}) to design the optimal coordination control for the systems (\ref{E6}).

\subsection{The Coupled HJ Equation}
Define the local performance indexes (cost functionals) by
\begin{align}\label{E7}
J_i(e_i(0),u_i,u_{(j)})=& \int_0^\infty r_i(e_i,u_i,u_{(j)})dt \nonumber \\
=&\int_0^\infty (e_i^T Q_{ii} e_i + \sum \limits_{j\in\mathcal{I}}
u_j^TR_{ij}u_j)dt \nonumber \\
=&\int_0^\infty (e_i^T Q_{ii} e_i + u_i^T R_{ii} u_i +\sum
\limits_{j\in\mathcal{I},j\neq i} u_j^TR_{ij}u_j)dt\nonumber \\
=&\int_0^\infty (e_i^T Q_{ii} e_i + u_i^T R_{ii} u_i +\sum
\limits_{j\in N_i} u_j^TR_{ij}u_j)dt,
\end{align}
with $r_i(e_i,u_i,u_{(j)})=e_i^T Q_{ii} e_i + u_i^T R_{ii} u_i+\sum
\limits_{j\in N_i} u_j^TR_{ij}u_j$. $u_{(j)}$ are the control input
vectors $\{u_j : j \in N_i\}$ of the neighbors of node $v_i$.

All weighting matrices are constant and satisfy
$Q_{ii}>0$, $R_{ii}> 0$ and $R_{ij}\geq 0$. Note that if $u_{j}$ is
the control inputs of the neighbors of node $v_i$, then $R_{ij}> 0$,
vice versa. Otherwise, $R_{ij}= 0$. In other words, the  performance
index $i$ depends on the input information of node $v_i$ and its
neighbors.
\begin{Prm} \label{P1}The problem required to be solved is that how to design the
local optimal coordination control to minimize the local performance
indexes (\ref{E7}) subject to (\ref{E6}) and make all nodes (agents)
reach a consensus on the control node (leader).
\end{Prm}
\begin{Def}[Admissible Coordination Control Policies]\cite{Vamvoudakisa2012}\label{D7}
Controls $u_i$ ($i\in \mathcal{I}$) are defined as admissible
coordination control policies if coordination controls $u_i$ ($i\in
\mathcal{I}$) not only stabilize the systems (\ref{E6}) on
$\Omega_i\in R^n$ locally, but also make the local
cost functional (\ref{E7}) finite.
\end{Def}

Under the given admissible coordination control policies $u_i$ and $u_{(j)}$, the local value function $V_i(e_i)$
for node $v_i$ is defined by
\begin{align}\label{E8}
V_i(e_i(t))
=& \int_t^\infty r_i(e_i,u_i,u_{(j)})dt \nonumber\\
=& \int_t^\infty (e_i^T Q_{ii} e_i + u_i^T R_{ii} u_i +\sum\limits_{j\in N_i} u_j^TR_{ij}u_j)dt,\nonumber\\
\end{align}
and the local coupled nonlinear Lyapunov equations for (\ref{E6}) are
\begin{align}\label{E9}
0=&H_i(e_i,V_{e_i},u_i,u_{(j)})\nonumber\\
\equiv& r_i(e_i,u_i,u_{(j)})+V_{e_i}^T((L_i+B_i)\otimes I_n)(f_e(t)+g(x)u)\nonumber \\
=&e_i^T Q_{ii} e_i + u_i^T R_{ii} u_i +\sum \limits_{j\in N_i} u_j^TR_{ij}u_j  + V_{e_i}^T \mathcal{L}_i(f_e(t)+g(x) u),
\end{align}
with $\mathcal{L}_i=(L_i+B_i)\otimes I_n$. $V_{e_i}$ is the partial
derivative of the value function $V_i(e_i)$ with respect to
$e_i$.

Meanwhile, the local coupled Hamiltonians of Problem \ref{P1} are
defined by
\begin{align}\label{E10}
H_i(e_i,V_{e_i},u_i,u_{(j)})=& r_i(e_i,u_i,u_{(j)})+V_{e_i}^T\mathcal{L}_i(f_e(t)+g(x)u)\nonumber \\
=&e_i^T Q_{ii} e_i + u_i^T R_{ii} u_i +\sum \limits_{j\in N_i}u_j^TR_{ij}u_j + V_{e_i}^T\mathcal{L}_i(f_e(t)+g(x) u).
\end{align}

According to the necessary condition of optimality principle, we can
obtain
\begin{align}\label{E13a}
u_i=&-\frac{1}{2}R_{ii}^{-1}\left(\frac{\partial u^T}{\partial u_i}\right)g^T(x)\mathcal{L}_i^TV_{e_i}\nonumber \\
=&-\frac{1}{2}R_{ii}^{-1}g_i^T(x_i)((l_{ii}+b_{ii})\otimes
I_n)^TV_{e_i}.
\end{align}

Assume that the local optimal value functions $V^*_i(e_i)$ satisfy
the coupled HJ equations
\begin{eqnarray}\label{E11}
\min \limits_{u_i} H_i(e_i,V^*_{e_i},u_i,u_{(j)})&=&0,
\end{eqnarray}
then, the local optimal coordination controls are
\begin{eqnarray}\label{E13}
u_i^*=-\frac{1}{2}R_{ii}^{-1}g_i^T(x_i)((l_{ii}+b_{ii})\otimes I_n)^TV^*_{e_i}.
\end{eqnarray}

Inserting $u^*_i$ and $u^*_{(j)}$ to (\ref{E9}), we can obtain
\begin{align*}
0=&e_i^TQ_{ii}e_i+u_i^{*T} R_{ii}u^*_i+\sum\limits_{j\in
N_i}u_j^{*T}R_{ij}u_j^*+V_{e_i}^{*T}\mathcal{L}_i(f_{e}(t)+g(x)u^*)\nonumber\\
=&e_i^TQ_{ii}e_i+\frac{1}{4}V_{e_i}^{*T}((l_{ii}+b_{ii})\otimes I_n)g_i(x_i)R_{ii}^{-1}g_i^T(x_i) ((l_{ii}+b_{ii})\otimes I_n)^TV^*_{e_i}\nonumber \\
&+\frac{1}{4}\sum \limits_{j\in N_i}V_{e_j}^{*T}((l_{jj}+b_{jj})\otimes I_n)g_j(x_j)R_{jj}^{-1}R_{ij}R_{jj}^{-1}g_j^T(x_j)((l_{jj}+b_{jj})\otimes I_n)^TV^*_{e_j}+V_{e_i}^{*T}\mathcal{L}_i(f_{e}(t)+g(x)u^*).
\end{align*}

We can rewrite it as the coupled HJ equations (see Appendix A)
\begin{align}\label{E14a}
0=&e_i^TQ_{ii}e_i+\frac{1}{4}V_{e_i}^{*T}((l_{ii}+b_{ii})\otimes I_n)g_i(x_i)R_{ii}^{-1}g_i^T(x_i)((l_{ii}+b_{ii})\otimes I_n)^TV^*_{e_i}\nonumber\\
&+\frac{1}{4}\sum \limits_{j\in N_i}\big( V_{e_j}^{*T}((l_{jj}+b_{jj})\otimes I_n)g_j(x_j)R_{jj}^{-1}R_{ij}R_{jj}^{-1}g_j^T(x_j)((l_{jj}+b_{jj})\otimes I_n)^TV^*_{e_j} \big) \nonumber\\
&+V_{e_i}^{*T}((l_{ii}+b_{ii})\otimes
I_n)(f_{ei}(t)+g_i(x_i)u_i^*)+V_{e_i}^{*T}\sum \limits_{j\in N_i}((l_{ij}+b_{ij})\otimes
I_n)(f_{ej}(t)+g_j(x_j)u_j^*).
\end{align}

Inserting (\ref{E13}) to (\ref{E14a}), we can get

\begin{align}\label{E141}
0=&e_i^TQ_{ii}e_i-\frac{1}{2}V_{e_i}^{*T}((l_{ii}+b_{ii})\otimes I_n)g_i(x_i)R_{ii}^{-1} g_i^T(x_i)((l_{ii}+b_{ii})\otimes I_n)^TV^*_{e_i}\nonumber\\
&+\frac{1}{4}\sum \limits_{j\in N_i}V_{e_j}^{*T}((l_{jj}+b_{jj})\otimes I_n) g_j(x_j)R_{jj}^{-1}R_{ij}R_{jj}^{-1} g_j^T(x_j)((l_{jj}+b_{jj})\otimes I_n)^TV^*_{e_j} \nonumber\\
&+V_{e_i}^{*T}\sum \limits_{j\in \{N_i, i\}}  ((l_{ij}+b_{ij})\otimes I_n) f_{ej}(t) -\frac{1}{2}V_{e_i}^{*T}\sum \limits_{j\in N_i} ((l_{ij}+b_{ij})\otimes I_n)g_j(x_j)R_{jj}^{-1} g_j^T(x_j)((l_{jj}+b_{jj})\otimes I_n)^TV^*_{e_j}.
\end{align}

Note that the optimal value functions $V^*_i(e_i)$ ($i =1,\ldots,N$)
are the solution of equations (\ref{E141}). The optimal coordination
controls (\ref{E13}) can be obtained by $V^*_i(e_i)$. In fact, the
solution of equations (\ref{E141}) is a Nash equilibrium. Their
relationship will be introduced in the next section.

\subsection{Nash Equilibrium}
First, according to \cite{Tijs2003}, we introduce the Nash
equilibrium definition for multi-player games.
\begin{Def} [Global Nash Equilibrium]\label{D8} An $N$-tuple of
control policies $\{u_1^*, u_2^*,\ldots, u_N^* \}$ is referred to as
a global Nash equilibrium solution for an $N$-player game (graph
$\mathcal{G}_x$) if for all $i \in \mathcal{I}$
\begin{align*}
J_i^* \triangleq& J_i(u_1^*,u_2^*,\ldots,u_i^*,\ldots,u_N^*)\nonumber\\
\leq& J_i (u_1^*,u_2^*,\ldots,u_i,\ldots,u_N^*),\quad(u_i\neq
u_i^*).
\end{align*}
The $N$-tuple of the local performance values $\{J_1^*, J_2^*,
\ldots, J_N^* \}$ is known as a Nash equilibrium of the $N$-player
game (graph $\mathcal{G}_x$).
\end{Def}

Then, two important facts are obtained by Theorem \ref{T1} below,
that is, the conclusions (I) and (II).
\begin{Thm}\label{T1} Let $V^*_i(e_i)>0\in
C^1$, $i \in \mathcal{I}$ be a solution to coupled HJ equations
(\ref{E141}), the optimal coordination control policies $u_i^*$
($i\in \mathcal{I}$) be given by (\ref{E13}) in term of  these
solutions $V^*_i(e_i)$. Then
\begin{description}
  \item[(I)] The local neighborhood consensus error systems (\ref{E6})
  are asymptotically stable.
  \item[(II)]  The local performance values $J^*_i(e_i(0), u^*_i, u^*_{(j)})$ are equal to $V^*_i(e_i(0))$, $i\in \mathcal{I}$; and $u_i^*$ and $u_{(j)}^*$ are in Nash equilibrium.
\end{description}
\end{Thm}
\begin{proof}
First, the conclusion (I) is proven. Under the conditions, the local
optimal value functions $V^*_i(e_i)>0$ satisfy (\ref{E141}) then
they also satisfy (\ref{E9}). Take the time derivative of
$V^*_i(e_i)$
\begin{align*}
\dot{V}^*_{i}(e_i)=& V^{*T}_{e_i}\dot{e}_i\\
=&V_{e_i}^{*T}\mathcal{L}_i(f_e(t)+g(x)u^*)\\
=&-e_i^T Q_{ii} e_i - u_i^{*T} R_{ii} u^*_i - \sum \limits_{j\in
N_i} u_j^{*T}R_{ij}u^*_j.
\end{align*}

Since $Q_{ii}>0, R_{ii}>0, R_{ij}> 0$ and
$\dot{V}^*_{i}(e_i)<0$. Therefore, ${V}^*_{i}(e_i)$ is a Lyapunov
function for $e_i$. Furthermore, the local neighborhood consensus
error system (\ref{E6}) is asymptotically stable.

The conclusion (II) is obvious, according to the definition of
performance index, value function and Definition \ref{D8}.
\end{proof}
\begin{Rem} In Theorem 1, the part (II) states the fact that the solution of the equation set (\ref{E141}) is the Nash equilibrium. Note that the solution of (\ref{E141}) is not unique. In general, there exist multiple Nash equilibrium. 
In fact, in ADP field, the obtained optimal solution is the local optimum \cite{Cox1998}. The globally optimal solution can not be obtained unless we explore the entire state space. However, in general, it is not possible.
\end{Rem}

Obviously, if only the coupled HJ equations (\ref{E141}) can be
solved, we will obtain the Nash equilibrium for multi-agent systems.
However, due to the nonlinear nature of the coupled HJ equations
(\ref{E141}), obtaining its analytical solution is generally
difficult. Therefore, in the next section, the policy
iteration algorithm is used to solve the coupled HJ equations.

\subsection{Policy Iteration (PI) Algorithm for the Coupled HJ
Equations}\label{sec5} In general, equations (\ref{E141}) are
difficult or impossible to be solved. In the field of ADP and
reinforcement learning, PI algorithm is usually used to obtain the
solution of the HJB equation. Similarly, we solves the coupled HJ
equations by PI algorithm, which relies on repeated policy evaluation
(e.g. the solution of (\ref{E9})) and policy improvement (the
solution of (\ref{E13a})). The iteration process is implemented until
the result of policy improvement no longer changes. If controls of
all the nodes ($i=1,\ldots,N$) do not change under the framework of PI
algorithm, then they are the solution (Nash equilibrium) of the
coupled HJ equations (\ref{E11}) or (\ref{E141}). However, it is
necessary that the initial local coordination control policies must
be admissible control policies in PI algorithm.

\textbf{Policy Iteration Algorithm:} Start with admissible initial policies $u^0_1,\ldots
,u^0_N$.

\begin{Step}\label{Step1} (Policy Evaluation) Given the $N$-tuple of policies
$u^k_1,\ldots ,u^k_N$, solve for $N$-tuple of costs
$V^k_1,\ldots,V^k_N$ using (\ref{E9})
\begin{eqnarray}\label{E15a}
 0=H_i(e_i,V^k_{e_i},u_i^k,u^k_{(j)}), \forall i=1,\ldots,N.
\end{eqnarray}
\end{Step}

\begin{Step} (Policy Improvement) Update the $N$-tuple of
control policies using (\ref{E13a})
\begin{align}\label{E16a}
&u^{k+1}_i=-\frac{1}{2}R_{ii}^{-1}g_i^T(x_i)((l_{ii}+b_{ii})\otimes
I_n)^TV^k_{e_i}, \forall i=1,\ldots,N.\nonumber\\
\end{align}
\end{Step}

Go to step \ref{Step1}.

It does not stops until $u_i$ converge to $u_i^*$, for $\forall i$.

Next, inspired by the linear result in \cite{Vamvoudakisa2012}, we give a theorem to state the convergence of the policy
iteration algorithm for nonlinear case.
\begin{Thm}\label{T1a}(Convergence of Policy Iteration Algorithm). Assume policies of all nodes $i$ are
updated at each iteration in PI algorithm. Then for small 
$\bar{\sigma}(R_{jj}^{-1}R_{ij})$ and big $\underline{\sigma}(R_{ii})$, $u_i$ converges to the Nash
equilibrium and for all $i$, and the value functions converge to the
optimal value functions $V^k_i \rightarrow V^*_i$.
\end{Thm}
\begin{proof}
By the following facts,
\begin{align*}
&H_i(e_i,V^{k+1}_{e_i},u_i^{k+1},u^{k+1}_{(j)})-H_i(e_i,V^{k}_{e_i},u_i^{k},u^{k}_{(j)})\nonumber\\&=\sum\limits_{j\in \{N_i,i\}}(u^{k+1}_j-u^k_{j})^TR_{ij}(u^{k+1}_j-u^k_{j})+2\sum\limits_{j\in \{N_i,i\}}(u^k_j)^TR_{ij}(u^{k+1}_j-u^k_{j})+\Theta_i,
\end{align*}
where $\Theta_i=\left(V_{e_i}^{k+1} \right )^T \sum\limits_{j\in \{N_i,i\}}
((l_{ij}+b_{ij})\otimes I_n) (f_{ej}(t)+g_j(x_j)u^{k+1}_j)-\left(V_{e_i}^{k} \right )^T \sum\limits_{j\in \{N_i,i\}}
((l_{ij}+b_{ij})\otimes I_n) (f_{ej}(t)+g_j(x_j)u^{k}_j)$, and
\begin{align*}
&H_i(e_i,V^{k+1}_{e_i},u_i^{k+1},u^{k+1}_{(j)})
-H_i(e_i,V^{k}_{e_i},u_i^{k},u^{k}_{(j)})=r_i(e_i,u^{k+1}_i,u^{k+1}_{(j)})-r_i(e_i,u^{k}_i,u^{k}_{(j)}) + \Theta_i,
\end{align*}
we can obtain
\begin{align*}
&r_i(e_i,u^{k+1}_i,u^{k+1}_{(j)})-r_i(e_i,u^{k}_i,u^{k}_{(j)})= \sum\limits_{j\in \{N_i,i\}}(u^{k+1}_j-u^k_{j})^TR_{ij}(u^{k+1}_j-u^k_{j})+2\sum\limits_{j\in
\{N_i,i\}} (u^k_j)^TR_{ij}(u^{k+1}_j-u^k_{j}).
\end{align*}

Since, at $k+1$ and $k$ steps, the time derivative of the local value function can be written respectively as $\dot{V}_i^{k+1}=-r_i(e_i,u^{k+1}_i,u^{k+1}_{(j)})$ and
let $\dot{V}_i^{k} = - r_i(e_i,u^{k}_i,u^{k}_{(j)})$, the above
expression can be rewritten as
\begin{align*}
\dot{V}_i^{k}-\dot{V}_i^{k+1}=&\sum\limits_{j\in \{N_i,i\}}(u^{k+1}_j-u^k_{j})^TR_{ij}(u^{k+1}_j-u^k_{j})+2\sum\limits_{j\in \{N_i,i\}} (u^k_j)^TR_{ij}(u^{k+1}_j-u^k_{j}).
\end{align*}
A sufficient condition for $\dot{V}_i^{k}-\dot{V}_i^{k+1}\geq0$ is
$$\Delta u_j^TR_{ij}\Delta u_j \geq-2 (u^k_j)^TR_{ij}\Delta u_j,  j \in \{N_i,i\},$$ where $\Delta
u_j=u^{k+1}_j-u^k_{j}$, or
\begin{eqnarray*}
\Delta u_j^TR_{ij}\Delta u_j\geq (V_{e_j}^{k-1})^T((l_{jj}+b_{jj})\otimes I_n)g_j(x_j)R_{jj}^{-1}R_{ij}\Delta u_j,\quad  j \in \{N_i,i\}.
\end{eqnarray*}
After taking norms, the above inequality will always hold in
case of
$$ \underline{\sigma}(R_{ij})\|\Delta u_j\|\geq (l_{jj}+b_{jj}) \bar{\sigma}(R_{jj}^{-1}R_{ij})\|V_{e_j}^{k-1}\|\beta_i,  j \in \{N_i,i\}.$$
$\underline{\sigma}(\cdot)$ is the operator which takes the minimum
singular value, and $\bar{\sigma}(\cdot)$ is the operator which
takes the maximum singular value. This holds if
$\bar{\sigma}(R_{jj}^{-1}R_{ij})=0$. By continuity, it also holds
for small values of $\bar{\sigma}(R_{jj}^{-1}R_{ij})$ and big values of $\underline{\sigma}(R_{ii})$.

By integrating the both sides of $\dot{V}^k_i \geq \dot{V}^{k+1}_i$,
it follows that $ V^k_i \geq V^{k+1}_i$ which shows that it is a
nonincreasing function bounded below by zero. Therefore $V^k_i$ is
convergent as $k \rightarrow \infty$. We can write $\lim_{k
\rightarrow \infty} V^k_i = V^{\infty}_i$. According to the
definition of the local value function (\ref{E8}), we have
\begin{align*}
V_i^k \geq& \int_t^\infty (e_i^T Q_{ii} e_i + u_i^{*T} R_{ii} u^*_i+\sum \limits_{j\in N_i} u_j^{*T}R_{ij}u^*_j)dt \equiv V_i^*,
\end{align*}
where $u_i^*$ and $u_{(j)}^*$ are the optimal coordination controls.
Let $k \rightarrow \infty$ then $V_i^\infty\geq V^*_i.$

Since $V^*_i \leq V^{\infty}_i$, the algorithm converges to $V^*_i$,
and obtains the solution to the coupled HJ equations, that is, the
cooperate Nash equilibrium.
\end{proof}
\begin{Rem} The proof of Theorem \ref{T1a} shows that when $j \in N_i$, the node $i$ should
weight the neighborhood control $u_j$ in its performance index $J_i$
relatively less than the node $j$ weights its own control $u_j$ in
$J_j$.  While $\underline{\sigma}(R_{ii})$ should be enough big such that $\underline{\sigma}(R_{ii})\|\Delta u_i\|\geq (l_{ii}+b_{ii}) \|V_{e_i}^{k-1}\|\beta_i$, as $j=i$. Therefore, it is necessary to select the proper
weighting matrices for the local performance in practice. In addition, $\|g_i(x_i)\|$ should have small upper bound.
\end{Rem}

In what follows, the generalized fuzzy hyperbolic models are used to
solve the optimal solution of the coupled HJ equations by adaptive
algorithm. This scheme develops a single-network adaptive
architecture for approximating the solution of (\ref{E11}) or
(\ref{E141}), under the framework of PI algorithm.

\section{GFHM-based Approximate Solutions of the Coupled HJ Equations}\label{sec6}

In this section, our intention is to develop an adaptive algorithm to
approximate the solution of a set of coupled HJ equations
(\ref{E141}). However, the coordination control of each node not
only depends on information of itself, but also depends on that of its
neighborhood nodes, so it is difficult to be handled.

According to Lemma \ref{L1} and Lemma \ref{L2}, we utilize the
generalized fuzzy hyperbolic model to estimate $V_i(e_i)$ for the
first time. We call it as generalized fuzzy hyperbolic critic
estimator (GFHCE), as follows:
\begin{eqnarray}\label{E14}
V_i(e_i)=\theta^T_i \tanh(\Phi_i \bar{e}_{i})+\zeta_i+\varepsilon_i,
\end{eqnarray}
where $\bar{e}_{i}$ is the generalized input variable of ${e}_{i}$;
$\theta_i=[\theta_{i1},\theta_{i2},\ldots,\theta_{im}]^T \in {R}^{m}$ is the unknown ideal weight vector for node $v_i$;$\Phi_i =\diag\{\phi_{ij}\}$ ($j=1,\ldots,m$); $\zeta_i$ is a constant scalar
for node $v_i$, and $\varepsilon_i$ is the GFHCE estimation error
for $V_i(e_i)$. Note that, since $V_i(e_i)$ satisfies the condition of Lyapunov
function, that is, $V_i(0)=0$, ${\zeta}_i$ is set to zero.

\begin{Rem}\label{R6}
Since ${\Phi_i}$ is nonlinear in the parameters (NLIP) in GFHM,
stability analysis of the parameter is cumbersome in some
applications. Fortunately, the GFHM can also be seen as a
three-layer neural network model \cite{Abdollahi2002}, whose
activation function is seen as $\tanh(\cdot)$. If weights ${\Phi_i}$
are fixed, then GFHM is linear in the parameters (LIP) on
$\theta_i$. Here, assume $\Phi_i$ is fixed to $I_{m}$ ($I_{m}$ is an identity matrix).
\end{Rem}

The derivative of the value function $V_i(e_i)$ with respect to
$e_i$ is
\begin{eqnarray}\label{E15}
V_{e_i}=\Lambda_i (\bar{e}_i) \theta_i+ \Delta \varepsilon_i,
\end{eqnarray}
where $\Lambda_i(\bar{e}_i) = [\partial \tanh(\bar{e}_i)/\partial
e_i ]^T$ and $\Delta \varepsilon_i = \partial \varepsilon_i/\partial
e_i$.

Let $\hat{\theta}_i$ be the estimate of $\theta_i$, then we have the
estimates of $V_i(e_i)$ and $V_{e_i}$, as follows:
\begin{eqnarray}\label{E16}
\hat{V}_i(e_i)=\hat{\theta}_i^T \tanh(\bar{e}_i)
\end{eqnarray}
and
\begin{eqnarray}\label{E17}
\hat{V}_{e_i}=\Lambda_i(\bar{e}_i) \hat{\theta}_i.
\end{eqnarray}

Then, the approximate Hamiltonian functions corresponding to
(\ref{E10}) can be derived, as follows:
\begin{align}\label{E18}
\mathfrak{e}_i&=H_i(e_i,\hat{\theta}_i, u_i, u_{(j)})  \nonumber\\
&=r_i(e_i, u_i, u_{(j)})+\hat{V}_{e_i}^T((L_i+B_i)\otimes I_n)(f_e(t)+g(x) u)\nonumber\\
&=e_i^T Q_{ii} e_i + u_i^T R_{ii} u_i + \sum
\limits_{j\in N_i} u_j^TR_{ij}u_j+ \hat{\varphi}_i(\hat{\theta}_i),
\end{align}
where $\hat{\varphi}_i(\hat{\theta}_i) =
({\Lambda}_i(\bar{e}_i)\hat{\theta}_i)^T \sum \limits_{j\in \{ N_i, i\}}((l_{ij} + b_{ij})\otimes I_n)(f_{ej}(t)+g_j(x_j)u_j)$.

Given any admissible coordination control policies $u_i$ and
$u_{(j)}$, the desired $\hat{\theta}_i$ can be selected to minimize the
squared residual error $E_i(\hat{\theta}_i)$
\begin{eqnarray}\label{E19}
E_i(\hat{\theta}_i)=\frac{1}{2}\mathfrak{e}_i^T\mathfrak{e}_i.
\end{eqnarray}

The weight adaptive updating laws for $\hat{\theta}_i$ can be
obtained by the gradient descent algorithm \cite{KyriakosG2010}, as
follows
\begin{align}\label{E20}
\dot{\hat{\theta}}_i
 =-a_i \sigma_i ( \sigma_i^T\hat{\theta}_i+r_i(e_i,u_i,u_{(j)})),
\end{align}
where $a_i > 0$ is the gain of the adaptive updating law for
$\hat{\theta}_i$. $\sigma_i = \Lambda_i^T(\bar{e}_i)\sum \limits_{j\in \{N_i, i\}}((l_{ij}+b_{ij})\otimes I_n)(f_{ej}(t)+g_j(x_j)u_j)$.

\begin{Rem}\label{R7} To make
$\hat{\theta}_i$ converge to the ideal value ${\theta}_i$, the
persistent excitation (PE) condition must be guaranteed. Therefore,
the coordination control policies are mingled with the probing
noise. In general, the hybrid coordination controls require to be
sufficiently rich signals, which contains as least $n/2$ distinct
nonzero frequencies
(see,\cite{Vamvoudakisa2012,KyriakosG2010,Ioannou2006}).
\end{Rem}

Inserting (\ref{E17}) into (\ref{E13a}), the admissible coordination
control policy can be expressed as
\begin{eqnarray}\label{E22}
\hat{u}_i=-\frac{1}{2}R_{ii}^{-1}g_i^T(x_i) ((l_{ii}+b_{ii})\otimes
I_n)^T \Lambda_i(\bar{e}_i) \hat{\theta}_i,
\end{eqnarray}
with the adaptive updating law (\ref{E20}) of $\hat{\theta}_i$ .

Next, we will give the detailed design procedure for solving
equations (\ref{E9}) and (\ref{E13a}) by fuzzy dynamic programming.
Note that the following procedure exists in every iteration step of
PI algorithm.

\textbf{Step 1:} GFHM are used as approximators to estimate the
solution (value functions) of (\ref{E11}). Therefore, it results in
the squired residual errors (\ref{E19}).

\textbf{Step 2:} To minimize the squired residual errors (\ref{E19}),
the gradient descent algorithm is used to obtain the  adaptive
updating laws (\ref{E20}) of the weight $\hat{\theta}_i$.

\textbf{Step 3:} Starting with the initial admissible weights
$\hat{\theta}_i=\theta_i^0$ ($i = 1, \ldots, N$),  the update does
not stop until the weight $\hat{\theta}_i$ converges
($\|\hat{\theta}_i-\theta_i\|< \varepsilon_{\theta_i}$,
$\varepsilon_{\theta_i}$ is the ideal evaluated error).

\section{Stability  Analysis}\label{sec7}
In this section, we give stability analysis for our scheme by the
proof of Theorem \ref{T2}. Before obtaining Theorem \ref{T2}, we
need the following preparations.

Inserting (\ref{E15}) into (\ref{E9}), the following formulation is
obtained
\begin{align*}
0=&H_i(e_i,\theta_i,u_i,u_{(j)})\nonumber\\
=&r_i(e_i,u_i,u_{(j)})+\theta_i^T
\Lambda_i^T(\bar{e}_i)\mathcal{L}_i(f_e(t)+g(x) u)+\Delta\varepsilon_i^T
\mathcal{L}_i(f_e(t)+g(x) u)\nonumber\\
=&r_i(e_i,u_i,u_{(j)})+\theta_i^T \sigma_i + \Delta\varepsilon_i^T
\mathcal{L}_i(f_e(t)+g(x) u).
\end{align*}
Further, we can obtain
\begin{eqnarray}\label{E23}
r_i(e_i,u_i,u_{(j)})+\theta_i^T \sigma_i=\varepsilon_{HJ_i},
\end{eqnarray}
where $\varepsilon_{HJ_i}=-\Delta\varepsilon_i^T
\mathcal{L}_i(f_e(t)+g(x) u)$. $\varepsilon_{HJ_i}$ is the residual
error resulting from the function approximation error.

Define the weight estimation error of the local value function as
$\tilde{\theta}_i=\hat{\theta}_i-\theta_i$. By (\ref{E20}) and
(\ref{E23}), we have
\begin{align}\label{E24}
\dot{\tilde{\theta}}_i=&\dot{\hat{\theta}}_i-\dot{\theta}_i=\dot{\hat{\theta}}_i\nonumber\\
=&-a_i\sigma_i(\sigma_i^T\hat{\theta}_i-\sigma_i^T\theta_i+\varepsilon_{HJ_i})\nonumber\\
=&-a_i\sigma_i(\sigma_i^T{\tilde\theta}_i+\varepsilon_{HJ_i}).
\end{align}

Throughout this section, the following assumptions should hold:
\begin{Asu}\label{Asu2}
\end{Asu}
  \begin{enumerate}
    \item The persistent excitation condition ensures $\sigma_{m_i}<\|\sigma_i\|<\sigma_{M_i}$; and $\|\theta_i\|<\theta_{M_i}$, where $\sigma_{m_i}$, $\sigma_{M_i}$ and $\theta_{M_i}$
  are positive constants;
    \item The error $\|\varepsilon_{HJ_i}\|$ in the coupled HJ equation has upper bound with $\|\varepsilon_{HJ_i}\|<\bar{\varepsilon}_i$, $\bar{\varepsilon}_i$
is a positive constant;
    \item The GFHCE estimation error $\varepsilon_i$ has upper bound with $\|\varepsilon_i\|\leq\varepsilon_{M_i}$; and
        $\|\Delta\varepsilon_i\|<\varepsilon_{\Delta{M_i}}$, $\varepsilon_{M_i}$ and $\varepsilon_{\Delta{M_i}}$ are also positive
constants.
  \end{enumerate}
\begin{Thm}\label{T2}
Consider multi-agent systems (\ref{E1}). Under the coordination
control policies (\ref{E22}) with the weight adaptive updating laws
of $\hat{\theta}_i$ (\ref{E20}), the local consensus errors $e_i$
and the weight estimation errors $\tilde{\theta}_i$ are UUB with the
bounds given by (\ref{E28}) and (\ref{E29}). Meanwhile
the control node trajectory $x_0(t)$ is CUUB, that is, all nodes
reach a consensus on $x_0(t)$. Moreover, the approximation
coordination control input $\hat{u}_i$ is close to the ideal
coordination control input $u_i$, i.e.,
$\|\hat{u}_i-u_i\|\leq\varepsilon_{u_i}$ as $t\rightarrow\infty$,
and $\varepsilon_{u_i}$ is a small positive constant.
\end{Thm}
\begin{proof} We choose the local Lyapunov functions candidate, as
follows
\begin{eqnarray}\label{E25}
L_i=L_{1_i}+L_{2_i},
\end{eqnarray}
where $L_{1_i}=\tr(\tilde{\theta}_i^T\tilde{\theta}_i)/2a_i $ and
$L_{2_i}=e_i^Te_i+2\Gamma_i V_i(e_i)$, with $\Gamma_i>0$.

According to Assumption \ref{Asu2} and (\ref{E24}), the time
derivative of the Lyapunov function candidate (\ref{E25})
is $$\dot{L}_i=\dot{L}_{1_i}+\dot{L}_{2_i},$$ where
\begin{align*}
\dot{L}_{1_i}=&\frac{1}{a_i}\tr(\tilde{\theta}_i^T\dot{\tilde{\theta}}_i)\nonumber\\
=&\frac{1}{a_i}(\tilde{\theta}_i^T(-a_i\sigma_i(\sigma_i^T\tilde{\theta}_i+\varepsilon_{HJ_i})))\nonumber\\
=&-\tilde{\theta}_i^T\sigma_i\sigma_i^T\tilde{\theta}_i-2\frac{a_i}{\sqrt{2a_i}}\tilde{\theta}_i^T\sigma_i\frac{1}{\sqrt{2a_i}}\varepsilon_{HJ_i}.\nonumber
\end{align*}
Also, since $\tilde{\theta}_i^T\sigma_i\sigma_i^T\tilde{\theta}_i >
0$, there exists $q_i>0$ such that $q_i \|\tilde{\theta}_i\|^2 \leq
\tilde{\theta}_i^T\sigma_i\sigma_i^T\tilde{\theta}_i$
($q_i\leq\|\sigma_i\|^2$). 
Then,
\begin{align}\label{E26}
\dot{L}_{1_i}\leq&-q_i\|\tilde{\theta}_i\|^2+\frac{a_i}{2}\|\sigma_i\|^2\|\tilde{\theta}_i\|^2+\frac{1}{2a_i}\bar{\varepsilon}_i^2\nonumber\\
\leq&(-q_i+\frac{a_i}{2}\|\sigma_i\|^2)\|\tilde{\theta}_i\|^2+\frac{1}{2a_i}\bar{\varepsilon}_i^2.
\end{align}
And
\begin{align}\label{E27}
\dot{L}_{2_i}=&2e_i^T\dot{e_i}+2\Gamma_i \dot{V}_i(e_i)\nonumber\\
=&2e_i^T\mathcal{L}_i(f_e(t)+g(x)u)-2\Gamma_ir_i(e_i,u_i,u_{(j)})\nonumber\\
=&2e_i^T\mathcal{L}_i(f_e(t)+g(x)u)-2\Gamma_i (e_i^T Q_{ii} e_i +
u_i^T R_{ii} u_i + \sum\limits_{j\in N_i} u_j^TR_{ij}u_j)\nonumber\\
=& \sum\limits_{j\in \{N_i,i\}} 2e_i^T((l_{ij}+b_{ij})\otimes
I_n)(f_{ej}(t)+g_j(x_j)u_j)-2\Gamma_i e_i^T Q_{ii} e_i -2\Gamma_i \sum\limits_{j\in \{N_i,i\}} u_j^TR_{ij}u_j \nonumber\\
=& \sum\limits_{j\in \{N_i,i\}} 2e_i^T((l_{ij}+b_{ij})\otimes I_n)f_{ej}(t)+\sum\limits_{j\in \{N_i,i\}} 2e_i^T((l_{ij}+b_{ij})\otimes
I_n)g_j(x_j)u_j\nonumber\\
&-2\Gamma_i e_i^T Q_{ii} e_i  -2\Gamma_i \sum\limits_{j\in \{N_i,i\}} u_j^TR_{ij}u_j\nonumber\\
\leq& 
\sum\limits_{j\in \{N_i,i\}}(\|((l_{ij}+b_{ij})\otimes I_n)\beta_j\|^2-2\Gamma_i \lambda_{\min}(R_{ij}))\|u_j\|^2+(2(\bar{N}_i+1)-2\Gamma_i \lambda_{\min}(Q_{ii}))\|e_i\|^2 \nonumber\\
&+ \sum\limits_{j\in\{N_i,i\}}\|((l_{ij}+b_{ij})\otimes I_n)f_{ej}(t)\|^2,
\end{align}
where $\bar{N}_i$ is the number of neighbors of node $i$.

Then, we can obtain
\begin{align*}
\dot{L}_i\leq&(-q_i+\frac{a_i}{2}\|\sigma_i\|^2)\|\tilde{\theta}_i\|^2+\frac{1}{2a_i}\bar{\varepsilon}_i^2+(2(\bar{N}_i+1)-2\Gamma_i \lambda_{\min}(Q_{ii}))\|e_i\|^2\nonumber\\
&+ \sum\limits_{j\in\{N_i,i\}}\|((l_{ij}+b_{ij})\otimes I_n)f_{ej}(t)\|^2+\sum\limits_{j\in \{N_i,i\}}(\|((l_{ij}+b_{ij})\otimes
I_n)\beta_j\|^2-2\Gamma_i \lambda_{\min}(R_{ij}))\|u_j\|^2.
\end{align*}

If $a_i$ and $\Gamma_i$ are selected to satisfy
\begin{align*}
&0<a_i<\frac{2q_i}{\|\sigma_{i}\|^2}, \nonumber\\
\end{align*}
\begin{align*}
&\Gamma_i > \max \left\{
\frac{\bar{N}_i+1}{\lambda_{\min}(Q_{ii})},
\left\{\frac{\|((l_{ij}+b_{ij}) \otimes
I_n)\beta_j\|^2}{2\lambda_{\min}(R_{ij})}, j \in \{{N}_i,i\}\right\}
\right\}
\end{align*}
and for large $\|\tilde{\theta}_i\|$ and $\|e_i\|$, the following inequalities 
\begin{align}\label{E28}
\|\tilde{\theta}_i\| >&\sqrt{ \frac{\sum\limits_{j\in \{N_i,i\}}\|((l_{ij}+b_{ij})\otimes I_n)f_{ej}(t)\|^2+\frac{1}{2a_i}\bar{\varepsilon}_i^2}{(q_i-\frac{a_i}{2}\sigma_{M_i}^2)}}
\triangleq b_{\tilde{\theta}_i}
\end{align}
or
\begin{align}\label{E29}
\|e_i\| >& \sqrt{\frac{\sum\limits_{j=\{N_i,i\}}\|((l_{ij}+b_{ij})\otimes I_n)f_{ej}(t)\|^2+\frac{1}{2a_i}\bar{\varepsilon}_i^2}{2\Gamma_i\lambda_{\min}(Q_{ii})-2(\bar{N}_i+1)}}
\triangleq b_{e_i}
\end{align}
hold, then $\dot{L}_i<0$. Therefore, using Lyapunov-like theorem
(see, \cite{Khalil2002} and \cite{Lewis1999}), it can conclude that
the local consensus errors $e_i$ and weight estimation errors
$\tilde{\theta}_i$ are uniformly ultimately bounded (UUB). Then
$x_0(t)$ is CUUB.

Next, we will prove $\|\hat{u}_i-u_i\|\leqslant\varepsilon_{u_i}$,
as $t\rightarrow \infty$. Recalling the expression of $u_i$ together
with (\ref{E13a}) and (\ref{E22}), we have
\begin{align}\label {E30}
\hat{u}_i-u_i=&-\frac{1}{2}R_{ii}^{-1}g_i^T(x_i)((l_{ii}+b_{ii})\otimes
I_n)^T(\hat{V}_{e_i}-V_{e_i}) \nonumber \\
=&-\frac{1}{2}R_{ii}^{-1}g_i^T(x_i)((l_{ii}+b_{ii})\otimes
I_n)^T(\Lambda_i(\bar{e}_i)\tilde{\theta}_i-\Delta\varepsilon_i).
\end{align}
When $t\rightarrow \infty$, the upper bound of (\ref{E30}) is
\begin{align*}
\|\hat{u}_i-u_i\| =&\frac{1}{2}\|R_{ii}^{-1}g_i^T(x_i)((l_{ii}+b_{ii})\otimes
I_n)^T(\Lambda_i(\bar{e}_i)\tilde{\theta}_i-\Delta\varepsilon_i)\|\\
\leq& 
\frac{1}{2}\|R_{ii}^{-1}g_i^T(x_i) ((l_{ii}+b_{ii})\otimes I_n)^T \|\sqrt{\|\Lambda_i(\bar{e}_i)\tilde{\theta}_i\|^2-2\Delta\varepsilon_i^T\Lambda_i(\bar{e}_i)\tilde{\theta}_i+\|\Delta\varepsilon_i\|^2}\\
\leq&\frac{1}{2}\|R_{ii}^{-1}g_i^T(x_i)
((l_{ii}+b_{ii})\otimes I_n)^T\| \sqrt{2\|\tilde{\theta}_i\|^2+2\|\Delta\varepsilon_i\|^2}\leq \varepsilon_{u_i},
\end{align*}
where $$\varepsilon_{u_i}=\frac{\sqrt{2}}{2}\|R_{ii}^{-1}g_i^T(x_i) ((l_{ii}+b_{ii})\otimes I_n)^T\|\sqrt{(b_{\tilde{\theta}_i}^2+\varepsilon_{\Delta
{M_i}}^2)}$$.
\end{proof}
\begin{Rem}
Since there exists $q_i \leq\|\sigma_i\|^2$, we can get $\frac{2q_i}{\|\sigma_i\|^2}\leq 2$. Let $\delta_i \in [0,2)$, then $\frac{2q_i}{\|\sigma_i\|^2}= 2-\delta_i$. Further, we can obtain $0<a_i<2-\delta_i $. Therefore, we can set the value $a_i$ by experience in the interval $(0,2)$.
\end{Rem}

\section{Simulation}
In this section, we illustrate the effectiveness of our scheme by a numerical example, and design the optimal coordination control by (\ref{E22}) for multi-agent systems.

Here, we consider the five-node digraph structure with leader node connected to node 3, such as Fig.\ref{F1}. The edge weights and the pinning gain in (\ref{E4}) were chosen as one.
\begin{figure}[!t]
\centering
\includegraphics[width=1.5in]{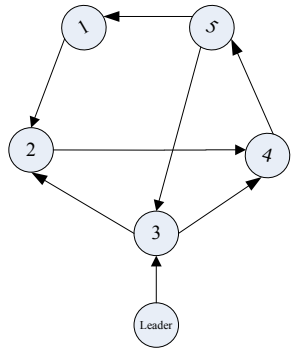}
\caption{The structure of five-node digraph with leader node}
\label{F1}
\end{figure}

For the structure in Fig.\ref{F1}, each node dynamic is considered as\\
Node 1:{
\begin{align*}
\dot{x}_{11}&=x_{12} - x_{11}^2x_{12}\\
\dot{x}_{12}&=-(x_{11}  + x_{12})(1-x_{11})^2 + x_{12}^2u_1,
\end{align*}}
Node 2:{
\begin{align*}
\dot{x}_{21}&=x_{22} - x_{21}^2x_{22}\\
\dot{x}_{22}&=-(x_{21}  + x_{22})(1-x_{21})^2 + 1.5x_{22}^2u_2,
\end{align*}}
Node 3:{
\begin{align*}
\dot{x}_{31}&=x_{32} - x_{31}^2x_{32}\\
\dot{x}_{32}&=-(x_{31}  + x_{32})(1-x_{31})^2 - 0.2x_{32}^2u_3,
\end{align*}}
Node 4:{
\begin{align*}
\dot{x}_{41}&=x_{42} - x_{41}^2x_{42}\\
\dot{x}_{42}&=-(x_{41}  + x_{42})(1-x_{41})^2 + 0.3x_{42}^2u_4,
\end{align*}}
Node 5:{
\begin{align*}
\dot{x}_{51}&=x_{52} - x_{51}^2x_{52}\\
\dot{x}_{52}&=-(x_{51}  + x_{52})(1-x_{51})^2 - 0.9x_{52}^2u_5,
\end{align*}}
While the state trajectory of the leader node is{
\begin{align*}
\dot{x}_{01}&=x_{02} - x_{01}^2x_{02}\\
\dot{x}_{02}&=-(x_{01}  + x_{02})(1-x_{01})^2.
\end{align*}}
Let $Q_{ii}=\left[
              \begin{array}{cc}
                1 & 0 \\
                0 & 1 \\
              \end{array}
            \right]$, $i=1,\ldots,5$; $R_{ii}=8.5$,
$i=1,\ldots,5$; $R_{ij}=0.1 (i\neq j)$, $i,j=1,\ldots,5$ (Note that
$R_{ij}=0$, if the ${v}_j$ is not the neighbor of ${v}_i$) and
$a_i=0.1$, $i=1,\ldots,5$. Here, for the sake of simplicity, we let
$\Phi=I_n$ and the generalized input variable be $e_i$.

Our intention is to design the optimal coordination control, making $x_i$ reach a consensus on leader and minimizing the cost functional (\ref{E7}). By our method, we can obtain an optimal coordination control
\begin{eqnarray}\label{E31}
 u_i^*=-\frac{1}{2}R_{ii}^{-1}g_i^T(x_i)(({l}_{ii}+b_{ii})\otimes I_n){\Lambda_i(e_i)} {\theta}^*_i,\nonumber\\ i=1,\ldots,5.
\end{eqnarray}

We can see the evolution of weight $\hat{\theta}_i$ from Fig. \ref{F2}. Obviously, after {10s}, $\hat{\theta}_i$ converge to the ideal values.

\begin{figure}[!t]
\centering
\includegraphics[width=3in]{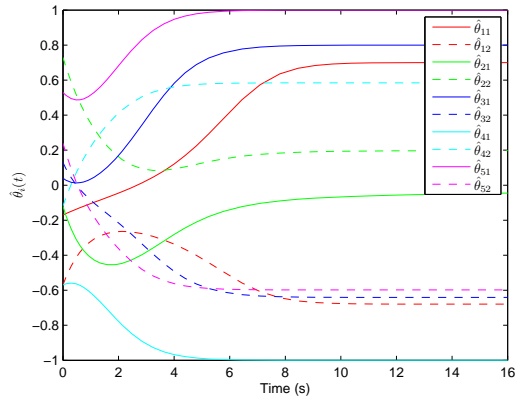}
\caption{{The convergence of GFHCE weight}} \label{F2}
\end{figure}

\begin{figure}[!t]
\centering
\includegraphics[width=3in]{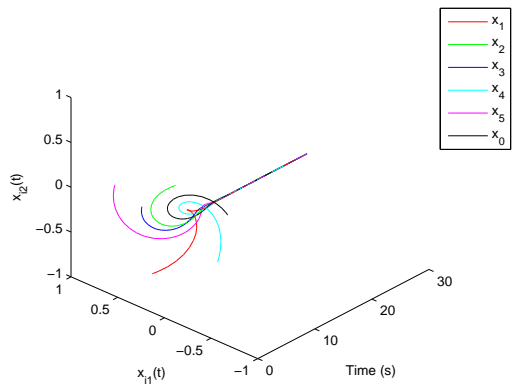}
\caption{{The evolution of the agent states}} \label{F3}
\end{figure}

\begin{figure}[!t]
\begin{center}
\begin{minipage}{160mm}
\centering
\subfigure[]{
\resizebox*{7cm}{!}{\includegraphics{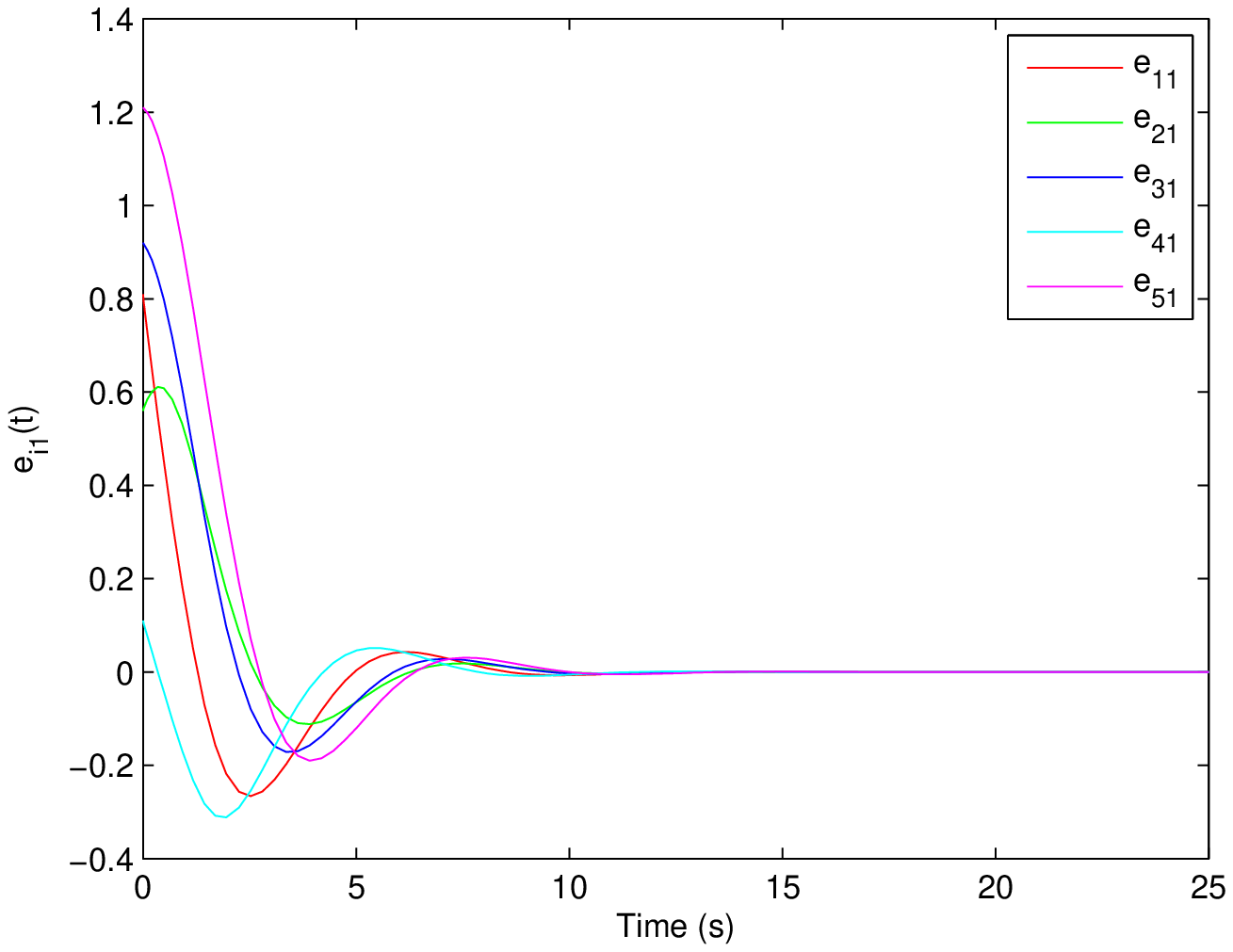}}}%
\centering
\subfigure[]{
\resizebox*{7cm}{!}{\includegraphics{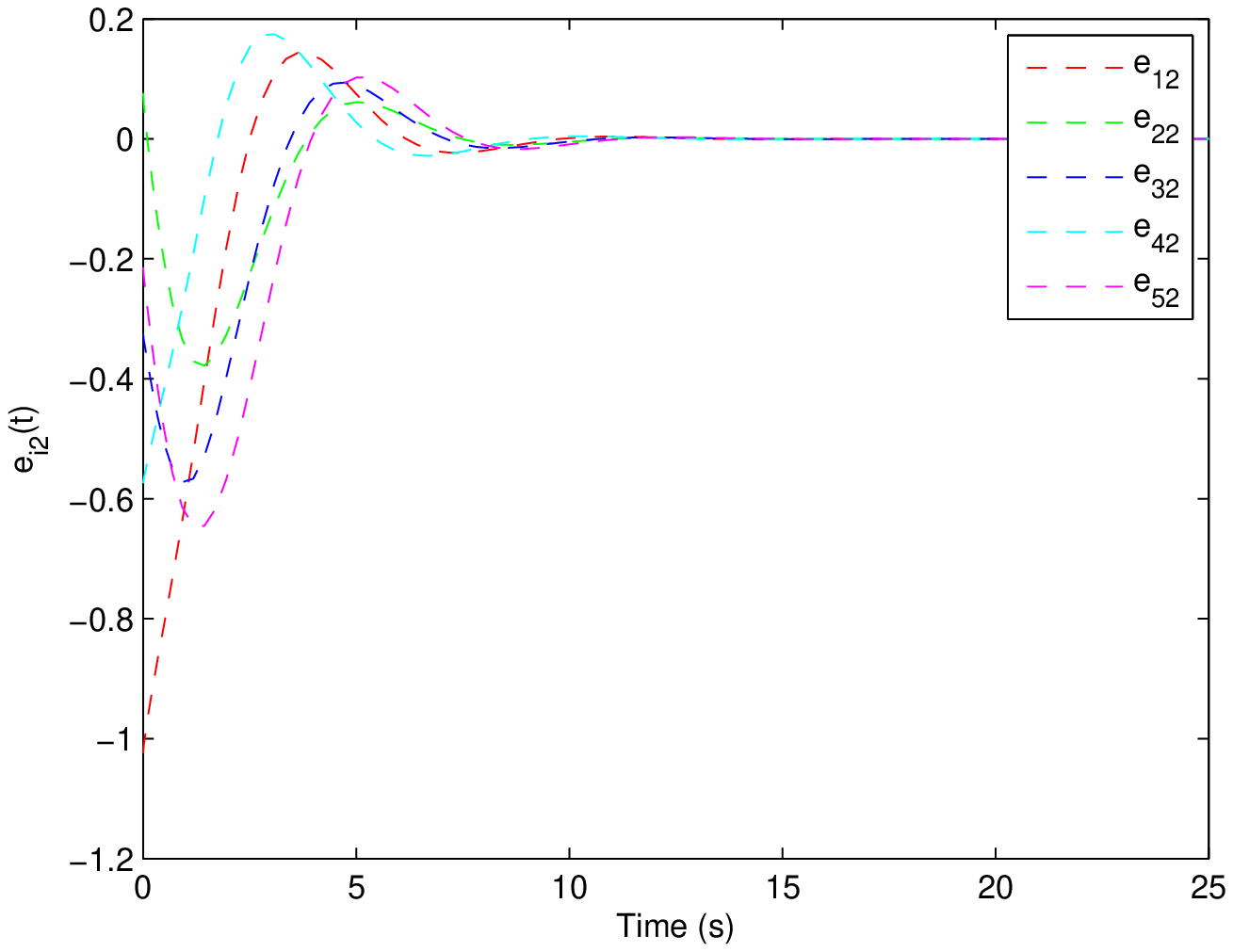}}}%
\end{minipage}
\end{center}
\caption{{The evolution of the local consensus error: (a) error $e_{i1}$ and (b) error $e_{i2}$.}}\label{F4}%
\end{figure}

Fig. \ref{F3} depicts the evolution of the every agent's state under the
optimal coordination controls (\ref{E31}) with the obtained ideal
value ${\theta}_i$. After {15s}, the state of each node reaches a
consensus on leader node.

Fig. \ref{F4} shows the evolution of the local consensus error under the
optimal coordination controls (\ref{E31}) with the obtained ideal
value ${\theta}_i$. After {15s}, the consensus error $e_i$ goes to
zero.
\begin{Rem}
Obviously, only one GFHM is used as the approximator for each
agent's critic network (value function), whereas the action network
presented in \cite{Zhang2011} and \cite{KyriakosG2010} is not
required. It eliminates the action network architecture. Therefore,
our scheme has greater advantage for multi-agent systems, comparing
with dual network method in \cite{Zhang2011,KyriakosG2010} and
\cite{Jagannathan2009}.
\end{Rem}

\section{Conclusion}
In this paper, the optimal coordination control has been presented
for multi-agent systems by fuzzy adaptive dynamic programming. Our
scheme is to approximate the solution of the coupled Hamilton-Jacobi
equation, making use of the single GFHM as approximator, rather than
using the dual-network model appeared in \cite{Zhang2011} and
\cite{KyriakosG2010}, under the framework of the PI algorithm. Then
the approximation solution has been utilized to obtain the optimal
coordination control. Since our method eliminates the action network
model and reduces the number of weights updated, FADP based on the
single GFHM has been a better scheme to design the optimal
coordination control for multi-agent systems. An example has been
presented to show the effectiveness of our scheme.

\section*{Acknowledgement}
The authors would like to thank the AE and reviewers for their valuable comments and suggestions that certainly improved the quality of the paper. We also thank the technical support from Zhenwei Liu and Ranran Li in the process of modifying. Simultaneously, we are also obliged to Dr. Feisheng Yang and Hongjing Liang for thier constructive help in checking our manuscript.

\appendices
\section{}
In order to prove the equation (\ref{E14a}) are equivalent to the
equation
\begin{align*}
0=&e_i^TQ_{ii}e_i+\frac{1}{4}V_{e_i}^{*T}((l_{ii}+b_{ii})\otimes I_n)g_i(x_i)R_{ii}^{-1}g_i^T(x_i)((l_{ii}+b_{ii})\otimes I_n)^TV^*_{e_i}\nonumber\\
&+\frac{1}{4}\sum \limits_{j\in N_i}V_{e_j}^{*T}((l_{jj}+b_{jj})\otimes I_n) g_j(x_j)R_{jj}^{-1}R_{ij}R_{jj}^{-1}g_j^T(x_j)((l_{jj}+b_{jj})\otimes I_n)^TV^{*}_{e_j}\nonumber\\
&+V_{e_i}^{*T}\mathcal{L}_i(f_{e}(t)+g(x)u^*),
\end{align*}
we have to deduce $V_{e_i}^{*T}((l_{ii}+b_{ii})\otimes I_n)(f_{ei}(t)+g_i(x_i)u_i^*)+V_{e_i}^{*T}\sum \limits_{j\in N_i}((l_{ij}+b_{ij})\otimes I_n)(f_{ej}(t)+g_j(x_j)u_j^*)$  in (\ref{E14a}) from the term $V_{e_i}^{*T}\mathcal{L}_i(f_{e}(t)+g(x)u^*)$. According to the previous definition of the paper, we know
$f_e(t)=    [ f_{e1}^T(t),  \cdots,   f_{ei}^T(t),  \cdots,    f_{ej}^T(t),  \cdots,   f_{eN}^T(t)]^T,$
$g(x)=\diag \{g_{1}(x_1),    \cdots,   g_{i}(x_i),   \cdots,    g_{j}(x_j),   \cdots,   g_{N}(x_N)\},$
$u^*=[u_1^{*T}, \cdots, u_i^{*T}, \cdots, u_j^{*T}, \cdots,
u_N^{*T}]^T, L_i=[l_{i1},\ldots,l_{ii}\ldots,l_{ij}\ldots,l_{iN}]$
and $B$ is a diagonal matrix, while let
$B_i=[b_{i1},\ldots,b_{ii}\ldots,b_{ij}\ldots,b_{iN}]$. Because
$\mathcal{L}_i=(L_i+B_i)\otimes I_n$, we can obtain
$\mathcal{L}_i=[(l_{i1}+b_{i1})\otimes
I_n,\ldots,(l_{ii}+b_{ii})\otimes I_n,\ldots,(l_{ij}+b_{ij})\otimes
I_n,\ldots,(l_{iN}+b_{iN})\otimes I_n]$. We also obtain
\begin{eqnarray*}
f_{e}(t)+g(x)u^*=
                       \left[
                         \begin{array}{c}
                           f_{e1}(t)+g_1(x_1)u_1^* \\
                           \vdots \\
                           f_{ei}(t)+g_i(x_i)u_i^* \\
                           \vdots \\
                           f_{ej}(t)+g_j(x_j)u_j^* \\
                           \vdots \\
                           f_{eN}(t)+g_N(x_N)u_N^* \\
                         \end{array}
                       \right].
\end{eqnarray*}
Therefore,
\begin{align}\label{AE35}
&V_{ei}^{*T}\mathcal{L}_i(f_{e}(t)+g(x)u^*)\nonumber\\
=&V_{ei}^{*T}((l_{i1} + b_{i1})\otimes
I_n)(f_{e1}(t)+g_1(x_1)u_1^*)+\ldots+ V_{ei}^{*T}((l_{ii} + b_{ii})\otimes I_n)(f_{ei}(t)+g_i(x_i)u_i^*)\nonumber\\
&+\ldots+ V_{ei}^{*T}((l_{ij} + b_{ij})\otimes I_n)(f_{ej}(t)+g_j(x_j)u_j^*)+\ldots +V_{ei}^{*T}((l_{iN} + b_{iN})\otimes I_n)(f_{eN}(t)+g_N(x_N)u^*_N)\nonumber\\
=&V_{ei}^{*T}\sum \limits_{j\in \mathcal{I}}((l_{ij}+b_{ij})\otimes
I_n)(f_{ej}(t)+g_j(x_j)u_j^*)\nonumber\\
=&V_{ei}^{*T}((l_{ii}+b_{ii})\otimes
I_n)(f_{ei}(t)+g_i(x_i)u_i^*)+V_{ei}^{*T}\sum \limits_{j\in \mathcal{I},i\neq
j}((l_{ij}+b_{ij})\otimes I_n)(f_{ej}(t)+g_j(x_j)u_j^*).
\end{align}

Note that the expression includes all nodes in graph
$\mathcal{G}_x$. When the node $v_j$ is not the neighbor of $v_i$,
$l_{ij}=0(i\neq j)$. Because $B=[b_{ij}]$ is a diagonal
matrix, $b_{ij}=0$, as $i\neq j$ (in order to understand easily, we retain
$b_{ij}$ item). Therefore, removing the items which are not neighbors of $v_i$, the expression (\ref{AE35})
can be rewritten as
\begin{align*}
&V_{ei}^{*T}\mathcal{L}_i(f_{e}(t)+g(x)u^*)=V_{e_i}^{*T}((l_{ii}+b_{ii})\otimes I_n)(f_{ei}(t)+g_i(x_i)u_i^*)+V_{e_i}^{*T}\sum \limits_{j\in N_i}((l_{ij}+b_{ij})\otimes I_n)(f_{ej}(t)+g_j(x_j)u_j^*).
\end{align*}
So, the {expression} (\ref{E14a}) holds.





\end{document}